\documentclass[12pt]{amsart}
\usepackage{amsmath, amssymb, amsthm}
\usepackage{enumitem}                                 
\usepackage[margin=1.4in]{geometry}
\usepackage{comment}

\usepackage{xcolor}
\usepackage[pagebackref]{hyperref}
\hypersetup{
   colorlinks,
    linkcolor={red!60!black},
    citecolor={blue!60!black},
    urlcolor={blue!90!black}
}

\newtheorem{theorem}{Theorem}[section]
\newtheorem{remark} [theorem]{Remark}

\newtheorem{Counter-example}[theorem]{Counter example}

\newtheorem{lemma}[theorem]{Lemma}
\newtheorem{proposition}[theorem]{Proposition}

\newtheorem{definition}[theorem]{Definition}
\newtheorem{corollary}[theorem]{Corollary}

\newtheorem*{theorem*}{Theorem}

\def\N{\mathbb N}
\def\R{\mathbb R}

\newcommand{\e}{\varepsilon}

\newcommand{\supp}{\text{supp}}

\title[Projections of self-similar measures]{On the dimension of orthogonal projections of self-similar measures}

\author{Amir Algom}
\address{Department of Mathematics, University of Haifa at Oranim, Tivon 36006, Israel}
\email{amir.algom@math.haifa.ac.il}

\author{Pablo Shmerkin}
\address{Department of Mathematics, the University of British Columbia. 1984 Mathematics Road, Vancouver BC V6T 1Z2, Canada}
\email{pshmerkin@math.ubc.ca}
\thanks{A.A. was supported by Grant No. 2022034 from the United States - Israel Binational Science Foundation  (BSF), Jerusalem, Israel.\newline
P.S. was supported by an NSERC Discovery Grant.}

\date{\today}

\begin{document}

\begin{abstract}
    Let $\nu$ be a self similar measure on $\mathbb{R}^d$, $d\geq 2$, and let $\pi$ be an orthogonal projection onto a $k$-dimensional subspace. We formulate a criterion on the action of the group generated by the orthogonal parts of the IFS on $\pi$, and show that it ensures the dimension of $\pi \nu$ is preserved; this significantly refines previous results by Hochman-Shmerkin (2012) and Falconer-Jin (2014), and is sharp for projections to lines and hyperplanes. A key ingredient in the proof is an application of a restricted projection theorem of Gan-Guo-Wang (2024).
\end{abstract}

\maketitle

\section{Introduction} \label{Section intro}
\subsection{Background}
Let $\Gamma_{d,k}$ be the Grassmannian of $k$-dimensional subspaces of $\mathbb{R}^d$. It  is a smooth compact manifold that supports a natural measure $\gamma_{d,k}$ invariant under the action of the orthogonal group $\mathbb{O}(\R^d)$. This measure is also locally equivalent to the push-forward of Lebesgue measure under the chart maps. There is a canonical identification between $\Gamma_{d,k}$ and the family of orthogonal projections of $\mathbb{R}^d$ onto $k$-dimensional subspaces. We will often apply this identification implicitly.

A  fundamental Theorem of Marstrand and Mattila is that for any Borel set $K\subseteq\R^d$,
\begin{equation} \label{Eq Marstrand}
    \dim \pi K = \min \lbrace k,\, \dim K \rbrace,\, \text{ for } \gamma_{d,k} \text{-a.e. } \pi \in \Gamma_{d,k}.
\end{equation}
This paper is about a significantly improved measure-theoretic version of this result for self-similar measures (defined below). The key to our approach is to combine two recent breakthroughs regarding refinements of \eqref{Eq Marstrand},  albeit, a priori, in very different directions. Let us recall them.

The first is the  recently developed theory of restricted projections, pioneered by F\"{a}ssler and Orponen \cite{Fassler2014Orponen}. The original conjecture formulated by these authors is that for $d=3$, the projection theorem \eqref{Eq Marstrand} remains true for Lebesgue a.e. element in a one-dimensional parameterization (of lines or planes) corresponding to a $C^2$ curve on the unit sphere $\mathbb{S}^2$, such that the vectors generated by the curve and its first two derivatives span the whole of $\mathbb{R}^3$. A model example of such a curve is
\begin{equation} \label{Eq non-deg exp}
    \gamma(\theta) := \frac{1}{\sqrt{2}} \left( \cos \theta, \sin \theta, 1 \right),\quad \theta \in [0, 2 \pi).
\end{equation}
The conjecture was first fully verified for  line projections corresponding to the curve \eqref{Eq non-deg exp} by K{\"a}enm{\"a}ki, Orponen and Venieri \cite{kaenmaki2017marstrand}; the case of line projections was later resolved in full generality by Pramanik, Yang and Zahl \cite{pramanik2022furstenberg}. The case of plane projections was then resolved by Gan, Guo, Guth, Harris, Maldague and Wang \cite{gan2022restricted}. A very general version in higher dimensions $d\geq 3$ was later given by Gan-Guo-Wang \cite{wang2022restricted}; this is the version we will use in the present paper. For more recent results regarding restricted projections and related research, see \cite{Harris2022proj, Harris2023proj, harris2022length, Gan2024guth, Eas2014keleti, Orponen2015proj, Orponen2020ven}.

The second direction is about improving \eqref{Eq Marstrand} by assuming $K$ supports a measure with additional arithmetic or dynamical structure (in some sense), and then employing tools from ergodic theory. Here, the aim is to gain information about \textit{specific} (rather than typical) directions. This approach, pioneered by Furstenberg \cite{furstenberg2008ergodic, furstenberg1970intersections} and Gavish \cite{Gavish2008scaling}, and later significantly developed by Hochman-Shmerkin \cite{hochman2009local} and Hochman \cite{hochman2010dynamics}, concerns measures $\mu$ that are \textit{uniformly scaling} in a suitable sense. Informally, this means that for $\mu$-a.e. $x$, the family of measures one sees as one ``zooms'' into $x$, all the while re-scaling and conditioning the measure, equidistributes towards some distribution $P$ (supported on Borel probability measures). Hochman-Shmerkin \cite[Theorem 1.10]{hochman2009local} showed  that this implies the existence of a  lower semicontinuous map $E:\Gamma_{d,k}\rightarrow\R$ such that $\dim \pi \nu \geq E(\pi)$ for every $\pi \in \Gamma_{d,k}$, that enjoys additional invariance properties (for the definition of the dimension of a measure see Section \ref{Section dimension}). Among the many applications of this approach was the resolution of Furstenberg's sumset conjecture \cite[Theorem 1.3]{hochman2009local}. It was later shown by Falconer-Jin \cite{Falconer2014Jin} that such  projection Theorems for self-similar measures can be derived without using the scenery flow but rather by purely symbolic arguments. This way they were also able to dispose of the separation conditions originally imposed by Hochman-Shmerkin and Furstenberg for their use of the scaling scenery. It has been recently shown, however,  that in fact all self-similar measures are uniformly scaling \cite{barany2023scaling, Aleksi2021flow}. Nonetheless, we will follow here the approach of Falconer-Jin \cite{Falconer2014Jin}, with suitable variants and simplifications. For more recent related results regarding projections of dynamically defined measure, see \cite{Bruce2019jin, Falconer2015survey, bruce2022furstenberg, Aleksi2024reso}.

\subsection{Main results}
Let us now discuss self-similar sets and measures, the main subject of this paper: Let $\Lambda = \lbrace f_i(x)=r_i\cdot h_i\cdot x+t_i \rbrace_{i=1} ^n$ be a finite family of non-singular contracting similarity maps of $\mathbb{R}^d$; this means that $0<r_i<1$, $h_i \in \mathbb{O}(\mathbb{R}^d)$, and $t_i \in\R^d$. We call $\Lambda$ an \textit{IFS} - Iterated Function System.    It is well known that there exists a unique compact set $\emptyset \neq K=K_\Lambda \subseteq\R^d$ such that
\begin{equation} \label{Eq union}
    K = \bigcup_{i=1} ^n f_i (K).
\end{equation}
The set $K$ is called  the \textit{attractor} of the IFS $\Lambda$. If the union \eqref{Eq union} is disjoint we say that the SSC - strong separation condition, is met.  

Next, let $\textbf{p}=(p_1,...,p_n)$ be a strictly positive probability vector, that is, $p_i >0$ for all $i$ and $\sum_{i=1}^n p_i =1$. We call $\mathbf{\Lambda} = \{ (f_i,p_i)\}_{i=1}^n$ a \emph{weighted IFS} or WIFS. Then there exists a unique Borel probability  measure $\nu = \nu_{\mathbf{\Lambda}}$,  with topological support $K$,  such that
\[
    \nu = \sum_{i=1} ^n p_i\cdot  f_i\nu,\quad \text{ where } f_i \nu \text{ is the push-forward of } \nu \text{ via } f_i.
\]
The measure $\nu$ is called a \textit{self-similar measure}. It can be obtained as the push-forward of the $p$-Bernoulli measure $\overline{\nu}$ on the space of infinite sequences $\lbrace 1,...,n \rbrace^\mathbb{N}$, via the \emph{coding map} $\Phi: \lbrace 1,...,n \rbrace^\mathbb{N} \rightarrow K$ defined by
\begin{equation} \label{eq:coding-map}
    \Phi(\omega) = \lim_{n\rightarrow \infty} f_{\omega_1} \cdot \cdot \cdot f_{\omega_n} (0).
\end{equation}

Now, one of the major applications of the method of Hochman-Shmerkin \cite{hochman2009local} described above is the following strong version of \eqref{Eq Marstrand} for self similar measures: Let $\nu$ be a self-similar measure with the SSC, and assume that for some (equivalently, any) $\pi \in \Gamma_{d,k}$,
\begin{equation} \label{condition HS}
    \big\{  h_{i_1}\cdots h_{i_m}\pi:\, i_1,....,i_m \in \lbrace1,...,n\rbrace,\, m\in \mathbb{N} \big\} \text{ is dense in } \Gamma_{d,k}.
\end{equation}
Then for any $\pi \in \Gamma_{d,k}$, the projection $\pi\nu$ is exact dimensional and
\[
    \dim \pi \nu = \min \lbrace k,\, \dim \nu \rbrace.
\]
As we previously mentioned, Falconer and Jin \cite{Falconer2014Jin} were later able to  remove all separation assumptions.

The main result of this paper is a significant refinement of this result, that is achieved by combining the ergodic theoretic approach of Hochman-Shmerkin and Falconer-Jin together with  restricted projection theory. Recall that a smooth curve $\gamma:(-\e,\e)\rightarrow\R^d$ is called non-degenerate if
\[
    \text{span} \lbrace \gamma'(\theta),...,\gamma^{(d)} (\theta)\rbrace =\R^d,\quad \text{ for all } \theta\in (-\e,\e).
\]
We require the following variant of this definition:
\begin{definition} \label{Def non-deg}
    Let $\Gamma\subseteq \Gamma_{d,k}$ be a subset of the Grassmannian. A smooth curve $\gamma:(-\e,\e)\rightarrow\R^d$ is called  \emph{$\Gamma$-adapted} if:
    \begin{enumerate}
        \item It is non-degenerate; and

        \item it satisfies
              \[
                  \textup{span} \lbrace \gamma'(\theta),...,\gamma^{(k)} (\theta)\rbrace \subseteq \Gamma\quad\text{for all }\theta\in (-\e,\e).
              \]
    \end{enumerate}

\end{definition}

Any closed subgroup $G$ of $\mathbb{O}(\mathbb{R}^d)$ acts on $\Gamma_{d,k}$ by left multiplication (or by right composition if we think of elements of $\Gamma_{d,k}$ as projections). We will denote by $G\cdot\pi$ the orbit of $\pi\in\Gamma_{d,k}$ under this action. It is always a smooth compact embedded submanifold of $\Gamma_{d,k}$ (see e.g. \cite[Theorem 3.29]{kirillov2008}; the immersion is an embedding in this case since $G$ is compact). Our main technical result is the following:
\begin{theorem} \label{Main Tech theorem}
    Let $\nu$ be a self-similar measure on $\mathbb{R}^d$ generated by a weighted IFS $\mathbf{\Lambda}$. Let $G=G_{\mathbf{\Lambda}}$ denote the closed group generated by the orthogonal parts of the elements of $\mathbf{\Lambda}$.

    Let $\pi \in \Gamma_{d,k}$. If there exists a $G\cdot \pi$-adapted curve, then $\pi\nu$ is exact dimensional, and
    \begin{equation} \label{eq:dim-preserved}
        \dim \pi \nu = \min \lbrace k,\, \dim \nu \rbrace.
    \end{equation}
\end{theorem}
Taking an arbitrary non-degenerate curve in $\mathbb{R}^d$, we see that Theorem \ref{Main Tech theorem} implies the results of Hochman-Shmerkin \cite{hochman2009local} and Falconer-Jin \cite{Falconer2014Jin} mentioned above as special cases. It also covers many new examples, as we now discuss.

For projections to lines and hyperplanes, we obtain sharp conditions for \eqref{eq:dim-preserved} to hold:
\begin{corollary} \label{cor:proj-dim-codim-1}
    Let $\nu$ be a self-similar measure on $\mathbb{R}^d$ generated by a weighted IFS $\mathbf{\Lambda}$. Let $G=G_{\mathbf{\Lambda}}$ denote the closed group generated by the orthogonal parts of the elements of $\mathbf{\Lambda}$.

    Let $\pi \in \Gamma_{d,1}$. Assume that the linear span of the lines $\{ g\pi: g\in G\}$ equals $\mathbb{R}^d$. Then
    \begin{align*}
        \dim \pi \nu         & = \min \lbrace 1,\, \dim \nu \rbrace,   \\
        \dim \pi^{\perp} \nu & = \min \lbrace d-1,\, \dim \nu \rbrace.
    \end{align*}
\end{corollary}
This result is sharp in the following  sense. Let $\{ h_i\}_{i=1}^n$ be orthogonal maps generating a subgroup whose closure is $G$. Let $\pi \in \Gamma_{d,1}$ and  let $W\subseteq \mathbb{R}^d$ be the linear span of the lines $\{ g\pi: g\in G\}$, and assume $\dim W \leq d-1$. Let $\Phi$ be an IFS comprised of maps of the form $\lbrace \lambda\cdot h_i\cdot x+t_i\rbrace_{i=1} ^{n}$, and suppose there are finite sets $U\subseteq W, V\subseteq W^\perp$ such that $U+ W=\lbrace t_i \rbrace$, and $|U|,|W|\geq 2$. If needed, we can increase $n$ by allowing for repetition between the $h_i$s, that is we can have $i\neq j$ and $h_i = h_j$. Thus, taking $\lambda>0$ sufficiently small, there exists a self-similar measure $\nu$ with the SSC such that $0<\dim \nu <1$, and $\nu=\mu_1\times \mu_2$ where the product structure is with respect to the direct sum $W\oplus W^\perp$, $\mu_1$ is a self-similar measure on $W$, $\mu_2$ a self-similar measure on $W^\perp$, both with positive dimension and the SSC. Let $\pi'$ be a line in $W$ (in particular, we can take $\pi' = \pi$). Then $\pi' \nu = \pi' \mu_1$ so $\dim \pi' \nu \leq \dim \mu_1 < \dim \nu$, whence \eqref{eq:dim-preserved}  fails.

An obvious candidate for the curve in Theorem \ref{Main Tech theorem} is the push-forward of a one-parameter subgroup of $G$. Using such curves, we obtain:
\begin{corollary} \label{cor:one-parameter}
    Let $\nu$ be a self-similar measure on $\mathbb{R}^{d}$ generated by a weighted IFS $\mathbf{\Lambda}$. Let $G=G_{\mathbf{\Lambda}}$ denote the closed group generated by the orthogonal parts of the elements of $\mathbf{\Lambda}$. Let $\mathfrak{g}$ be its Lie algebra.

    Suppose $A\in\mathfrak{g}$ and $v\in\mathbb{R}^d$ is a cyclic vector for $A$ (i.e. $v$ is not contained in any proper $A$-invariant subspace). Then, for any $1\le k\le d-1$, we have
    \[
        \dim\pi_k \nu = \min\{k,\dim\nu\},
    \]
    where $\pi_k = \textup{span}\{v,Av,...,A^{k-1}v\}$.
\end{corollary}
Note that we can state the assumption in terms of the group $G$ itself: if $v$ is a cyclic vector for $g\in G^{\circ}$ where  $G^{\circ}$ is the  identity component, then it is also a cyclic vector for any logarithm $A$ of $g$ in the Lie algebra. We also remark that some connected subgroups of $\mathbb{O}(\R^d)$ do not contain any matrices with cyclic vectors (e.g. if they fix a subspace of dimension $2$). Nonetheless, many of them do, while being far from acting transitively on $\Gamma_{d,k}$. So, this corollary again provides many new examples not covered by \cite{hochman2009local, Falconer2014Jin}.

As a specific example, consider the case that $G\subset\mathbb{O}(\R^4)$ is (or contains) a maximal torus (i.e. independent rotations around two orthogonal $2$-planes). After an orthogonal change of coordinates, the Lie algebra $\mathfrak{g}$ consists of matrices of the form
\[
    \begin{pmatrix}
        0         & -\lambda_1 & 0         & 0          \\
        \lambda_1 & 0          & 0         & 0          \\
        0         & 0          & 0         & -\lambda_2 \\
        0         & 0          & \lambda_2 & 0
    \end{pmatrix},\quad \lambda_i\in\R.
\]
Any vector $x=(x_1,x_2,x_3,x_4)$ with $(x_1,x_2),(x_3,x_4)\neq (0,0)$ is cyclic for all such matrices with $\lambda_i\neq 0$, $i=1,2$. Thus, for any such $x$ and any $\lambda\neq 0$, the $2$-planes $\pi=\pi_{x,\lambda}$ of the form
\[
    \textup{span}\{ (x_1,x_2,x_3,x_4), (- x_2,x_1, -\lambda x_4, \lambda x_3)  \}
\]
satisfy $\dim\pi\nu = \min\{2,\dim\nu\}$. It is not too hard to check that all such planes form a $3$-dimensional submanifold of the $4$-dimensional space $\Gamma_{4,2}$. We do not know whether dimension preservation holds for an open subset of $\Gamma_{4,2}$ in this case.

\subsection{Proof sketch and organization} After some brief preliminaries in Section  \ref{Section dimension}, we proceed to prove a measure theoretic version of the Gan-Guo-Wang projection Theorem in Section \ref{Section GGW}. The key ingredient here is the main technical result from \cite{wang2022restricted}, Theorem \ref{Coro eq. 2.7}. Then, in Section \ref{Section EFD}, we turn to projections of self-similar measures, and prove that the dimension of their projections are almost constant along $G$-orbits. This is Theorem \ref{thm:constant-dim}, which appears to be new. The proof relies heavily on local entropy averages, similarly to Hochman-Shmerkin \cite{hochman2009local} and Falconer-Jin \cite{Falconer2014Jin}. We then combine these two tools to prove Theorem \ref{Main Tech theorem} in Section \ref{Section proof}, where we also prove Corollary \ref{cor:one-parameter} and \ref{cor:proj-dim-codim-1}.

\subsection*{Acknowledgements} We thank Hong Wang for helpful discussions on restricted projections, and Meng Wu for his comments on an earlier version of this project.

\section{Preliminaries on the dimension theory of measures} \label{Section dimension}
For a Borel set $A$ in some metric space $X$, we denote by $\dim A$ its Hausdorff dimension, and by $\mathcal{P}(X)$ the space of Borel probability measures supported on $X$. Now,  let $\mu \in \mathcal{P}(X)$.   For every $x\in \supp(\mu)$  we define the lower local (pointwise) dimension of $\mu$ at $x$  as
\begin{equation*}
    \dim(\mu,x)=\liminf_{r\rightarrow 0} \frac{\log \mu (B(x,r))}{\log r}
\end{equation*}
where $B(x,r)$ denotes the closed ball or radius $r$ about $x$. The Hausdorff dimension of $\mu$ is defined as (and is also equal to)
\begin{equation} \label{Eq lower dim}
    \dim \mu := \text{ess-inf}_{x\sim \mu} \dim(\mu,x) = \inf \lbrace \dim A:\, \mu (A)>0 \rbrace.
\end{equation}
See \cite[Chapter 10]{falconer1997techniques} for more details about the second equality in \eqref{Eq lower dim}.

We note that $\dim$ is a Borel function of $\mu$ in the weak$^*$ topology. Indeed, by approximating indicators of balls by continuous functions, we have that $\dim(\mu,x)$ is Borel for each $x$. Then note that
\[
    \dim\mu = \inf_f \frac{\int f(x) \dim(\mu,x)d\mu}{\int f(x) d\mu},
\]
where the infimum is taken over a countable uniformly dense subset of $C^{+}(U)$, where $U$ is a neighborhood of $\supp(\mu)$, and only functions for which the denominator is positive are considered.

If $\dim (\mu,x)$ exists as a limit at almost every point, and is constant almost surely, we shall say that the measure $\mu$ is exact dimensional. In this case, most metric definitions of the dimension of $\mu$ coincide (e.g. lower and upper Hausdorff dimension and packing dimension). It is  known that self-similar measures are always exact-dimensional \cite{feng2009dimension}.

In this paper we call a measure $\mu\in \mathcal{P}(X)$  an $\alpha$-Frostman measure if   there is some $C>0$ such that
\[
    \mu(B(x,r))\leq C\cdot r^\alpha
\]
for every $x\in \supp(\mu)$ and $r>0$.  It is well known that a self-similar measure with respect to a strongly separated IFS $\mathbf{\Lambda}$ such that $\dim \nu = \dim K_\mathbf{\Lambda}$  is a $\dim K_\mathbf{\Lambda}$-Frostman measure. In general, exact dimensional measures (even self-similar measures) may fail the Frostman condition, but we do have the following useful standard property (which follows from Egorov's Theorem):
\begin{lemma} \label{lem:exact dim}
    Let $\mu \in \mathcal{P}(\mathbb{R}^d)$ be a measure with $\dim \mu = \alpha>0$. Then for all $n\in \mathbb{N}$ there exists a set $E_n \subseteq\R^d$ such that:
    \begin{enumerate}
        \item $\mu (E_n)\geq 1- \frac{1}{n}$.

        \item There exists $C=C(n,\mu)\in\R_+$ such that
              $$ \mu(B(x,r)\cap E_n)\leq C\cdot r^{\alpha-\frac{1}{n}},\, \text{ for all } x\in E_n \text{ and } r>0.$$
    \end{enumerate}
\end{lemma}

\section{A measure theoretic version of the Gan-Guo-Wang projection Theorem} \label{Section GGW}
%

Let $\gamma:(-\e,\e)\rightarrow\R^d$ be a smooth non-degenerate curve in the sense that
\[
    \det[ \gamma'(\theta),\ldots,\gamma^{(d)} (\theta)] \neq 0,\quad \text{ for all } \theta\in (-\e,\e).
\]
For $1\leq k \leq d$ define the $k$-th order tangent (or osculating) space of $\gamma$ at $\theta$ by
\[
    \Pi_\theta ^{(k)} :=\text{span}\{ \gamma'(\theta),\ldots,\gamma^{(k)} (\theta)\},
\]
where, as usual, we identify a plane with the orthogonal projection to it. The main result of Gan-Guo-Wang \cite{wang2022restricted}, that was already mentioned in the introduction, is a higher dimensional version of the restricted projections conjecture:
\begin{theorem}[{\cite[Theorem 1.2]{wang2022restricted}}] \label{Theorem Gan}
    Let $d\geq 2, 1\leq k \leq d$, and let $\gamma:(-\e,\e)\rightarrow\R^d$ be a smooth non-degenerate curve. Let $E\subseteq [0,1]^d$ be a Borel measurable set. Then
    \begin{equation}
        \dim \left( \Pi_\theta ^{(k)} \left( E \right) \right) = \min \lbrace \dim(E),\,k\rbrace,\quad \text{ for a.e. }\theta\in (-\e,\e).
    \end{equation}

\end{theorem}

Notice that Theorem \ref{Theorem Gan} is proved for sets. We prove a measure theoretic version, via modest modifications to the original proof of Gan-Guo-Wang:
\begin{theorem} \label{Theorem Gan measures}  Let $d\geq 2, 1\leq k \leq d$, and let $\gamma:(-\e,\e)\rightarrow\R^d$ be a non-degenerate curve. Let $\mu\in  \mathcal{P}\left([0,1]^d\right)$. Then
    \[
        \dim \left( \Pi_\theta ^{(k)}  \mu  \right) = \min \lbrace \dim(\mu),\,k\rbrace,\quad \text{ for a.e. }\theta\in (-\e,\e).
    \]
\end{theorem}

\subsubsection{Proof of Theorem \ref{Theorem Gan measures}}

We proceed to prove Theorem \ref{Theorem Gan measures}. We begin with the case of Frostman measures. The general case will then follow from Lemma \ref{lem:exact dim}. So,
let $\gamma$ be a non-degenerate curve. For $\delta>0$ we say that $\Lambda_\delta \subseteq (-\e,\e)$ is a $\delta$-net of $(-\e,\e)$ if
$$\delta\leq |\theta_1 - \theta_2 | <2\delta,\quad \forall \theta_1\neq \theta_2 \in \Lambda_\delta,$$
and for every $\theta \in (-\e,\e)$ there exists $\theta'\in \Lambda_\delta$ s.t. $\delta\leq |\theta - \theta'|< 2\delta$.

In what follows, we write $\mu_\theta=\Pi_\theta^{(k)}\mu$ for simplicity.



Here is the main technical result of \cite{wang2022restricted} that we will use as a black box:

\begin{theorem}[{\cite[Eq. (2.7)]{wang2022restricted}}] \label{Coro eq. 2.7}
    Let $\alpha\in (0,d), \alpha^* \in (0, \, \min\lbrace k,\, \alpha \rbrace)$. Let $\mu \in \mathcal{P}(B^d (0,1))$ be an $\alpha$-Frostman measure. Let $\delta>0$ be small, and let $\Lambda_\delta$ be a $\delta$-net. For every $\theta \in \Lambda_\delta$ let $\mathbb{D}_\theta$ be a disjoint collection of at most $\delta^{-\alpha^*}$ balls of radius $\delta$ in $\Pi_\theta ^{(k)} \left(\R^d \right)$, and let $\mu_\theta:= \Pi_\theta ^{(k)} \mu$.

    Then there exists $\eta_0 = \eta_0 (\alpha,\alpha^*)>0$ such that
    \[
        \delta \sum_{\theta \in \Lambda_\delta} \mu_{\theta} \left( \bigcup_{D\in \mathbb{D}_\theta}  D \right) \lesssim_{\alpha,\alpha^*} \delta^{\eta_0},
    \]
    where the implicit constant depends on $\alpha,\alpha^*$ but not on $\delta$.
\end{theorem}

\begin{proof}[Proof of Theorem \ref{Theorem Gan measures} for Frostman measures]

    Let $\mu$ be an $\alpha$-Frostman measure; without loss of generality, $\alpha\le k$. Let  $\eta0$ be small, and define $\Theta \subseteq [0,1]$ via
    \[
        \Theta:=\left\lbrace \theta\in (-\e,\e):\, \dim \left( \Pi_\theta ^{(k)} \mu \right) < \alpha-2\eta\right\rbrace.
    \]
    By $\sigma$-stability of dimension, it is enough to show that $\mathcal{H}^1(\Theta)=0$. Suppose, for the sake of contradiction, that $\mathcal{H}^1 (\Theta)>0$.

    For every $\theta \in \Theta$, there exists a set $K_\theta \subseteq \Pi_\theta ^{(k)} $ such that $\dim K_\theta < s$, and $c_\theta:=\Pi_\theta ^{(k)} \nu (K_\theta) >0$. Since $\Theta = \cup_{n\in\mathbb{N}} \{ c_\theta>1/n\}$, it follows from the sub-additivity of Hausdorff measure for general sets that there exists a subset $\Theta'\subseteq \Theta$ such that $\mathcal{H}^1(\Theta')>0$, and there is a constant $c>0$ such that $c_\theta \geq c$ for all $\theta\in \Theta'$. Without loss of generality, we assume $\Theta'=\Theta$ and $c=1$.

    Let $\delta>0$ be small. For every $\theta \in \Theta$, let
    \[
        \left\lbrace B\left( \Pi_\theta ^{(k)} (x_j(\theta)),\, r_j(\theta) \right) \right \rbrace_j,\quad x_j(\theta) \in\R^d
    \]
    be a covering of $K_{\theta}$ by discs of dyadic radii smaller than $\delta$, such that
    \begin{equation} \label{Eq 2.9}
        \sum_j r_j(\theta)^{s+\eta} <1.
    \end{equation}
    For each $\theta \in \Theta$ and $k\geq |\log_2 \delta|$, let
    \[
        \mathbb{D}_{\theta, k}:= \left\lbrace B\left( \Pi_\theta ^{(k)} (x_j(\theta)),\, r_j(\theta) \right):\, r_j(\theta)=2^{-k} \right \rbrace,
    \]
    and
    \[
        \mathbf{D}_{\theta,k}:= \bigcup_{D\in \mathbb{D}_{\theta,k}} D.
    \]
    Then for each $\theta \in \Theta$ we have
    \[
        \sum_{k\geq |\log_2 \delta|} \mu_\theta \left( \mathbf{D}_{\theta,k} \right)  \geq  1.
    \]
    By the Vitali covering Lemma, for each $\theta \in \Theta$ and each $k$, we can find a disjoint sub-collection $\mathbb{D}_{\theta, k} ' \subseteq \mathbb{D}_{\theta, k}$ such that
    \[
        1 \leq \sum_{k\geq |\log_2 \delta|} \mu_\theta \left( \mathbf{D}_{\theta,k} '' \right),
    \]
    where
    \[
        \mathbf{D}_{\theta,k} '':= \bigcup_{D'\in \mathbb{D}_{\theta,k} '} \left( 10^d D' \right).
    \]

    Thus,
    \[
        \mathcal{H}^1 (\Theta) \leq \sum_{k\geq |\log_2 \delta|} \int_\Theta \mu_\theta \left( \mathbf{D}_{\theta,k} '' \right)\, d\theta.
    \]

    By \eqref{Eq 2.9} for each $\theta$ and $k$ the set $\mathbb{D}_{\theta,k} '$ is a union of at most $2^{k(s+\eta')}$ disjoint discs of radius $2^{-k}$. Applying Theorem \ref{Coro eq. 2.7} and the triangle inequality, we can find $\eta_0>0$ independent of $k$ and $\delta$ such that
    \[
        \mathcal{H}^1 (\Theta)\lesssim \sum_{k\geq |\log_2 \delta|} 2^{-k \eta_0}.
    \]
    Letting $\delta \rightarrow 0$, we see that $\mathcal{H}^1(\Theta)=0$, which is the desired contradiction.
\end{proof}

\begin{proof}[Proof of Theorem \ref{Theorem Gan measures}]
    Consider now the case of general $\mu$. We may clearly assume $\dim \mu = \alpha>0$ - otherwise the result is trivial. By Lemma \ref{lem:exact dim},  for all $n\in \mathbb{N}$ there exists a set $E_n \subseteq\R^d$ such that:
    \begin{enumerate}
        \item $\mu (E_n)\geq 1- \frac{1}{n}$.

        \item There exists $C=C(n,\mu)\in\R_+$ such that
              $$ \mu(B(x,r)\cap E_n)\leq C\cdot r^{\alpha-\frac{1}{n}},\, \text{ for all } x\in E_n \text{ and } r>0.$$
    \end{enumerate}
    Let $1\leq k \leq d$.  Since for all $n$ the measure $\mu|_{E_n}$ is an $\alpha-\frac{1}{n}$-Frostman measure, applying Theorem \ref{Theorem Gan measures} for Frostman measures to it we get that
    $$X_n := \lbrace \theta\in (-\e,\e):\, \dim \left( \Pi_\theta ^{(k)} \left( \mu|_{E_n} \right) \right) = \min \lbrace \dim(\mu|_{E_n}),\,k\rbrace \rbrace$$
    has full Lebesgue measure in $(-\e,\e)$ for all $n$. So,
    $$X:=\liminf X_n \subseteq (-\e,\e)$$
    has full Lebesgue measure in $(-\e,\e)$.

    Finally, let $\theta \in X$. Let $A\subseteq\R^k$ be such that
    \[
        \Pi_\theta ^{(k)} \left( \mu \right) \left( A \right) >0.
    \]
    Then for every $n$ such that $\mu(E_n)>1-\mu(A)$ we have
    \[
        \Pi_\theta ^{(k)} \left( \mu|_{E_n} \right) \left( A \right) >0.
    \]
    Since $\theta \in X$ then for all $n$ large enough $\theta \in X_n$. Let $n_0$ be such that it satisfies both of these conditions. Then the previously displayed equation holds and we also we  have $\theta \in X_{n_0}$ and so
    \[
        \dim \Pi_\theta ^{(k)} \left( \mu|_{E_{n_0}} \right)  = \min \lbrace \dim(\mu|_{E_{n_0}}),\,k\rbrace  = \min \lbrace \dim(\mu),\,k\rbrace.
    \]
    Thus,
    \[
        \dim A \geq \dim \Pi_\theta ^{(k)} \left( \mu|_{E_{n_0}} \right) =\min \lbrace \dim(\mu),\,k\rbrace.
    \]
    It follows that
    \[
        \dim \Pi_\theta ^{(k)} \left( \mu \right)  = \inf \lbrace \dim A:\, \Pi_\theta ^{(k)} \left( \mu \right) (A)>0 \rbrace \geq \min \lbrace \dim(\mu),\,k\rbrace,
    \]
    and so
    \[
        \dim \Pi_\theta ^{(k)} \left( \mu \right) = \min \lbrace \dim(\mu),\,k\rbrace.
    \]
    Since this holds for every $\theta \in X$ which carries full Lebesgue measure, the proof of Theorem \ref{Theorem Gan measures} is now complete.
\end{proof}

\section{Projections of self-similar measures and \texorpdfstring{$G$}{G}-orbits} \label{Section EFD}

\subsection{Main result: almost constancy along  \texorpdfstring{$G$}{G}-orbits}
As we discussed in Section \ref{Section intro},  the structure of self-similar measures allows for the application of ergodic theoretic tools in geometric problems. In this section, we use such tools to prove  ``almost constancy'' of the dimension of projections of self-similar measures along $G$-orbits. For any closed subgroup $G\leq \mathbb{O}(\mathbb{R}^d)$ let $\lambda_G$ denote its (normalized) Haar measure. Here is the main result of this Section:
\begin{theorem} \label{thm:constant-dim}
    Let $\nu$ be a self-similar measure on $\mathbb{R}^d$ generated by a weighted IFS $\mathbf{\Lambda}$. Let $G=G_{\mathbf{\Lambda}}$ denote the closed group generated by the orthogonal parts of the elements of $\mathbf{\Lambda}$.

    Then for any $\pi\in G_{d,k}$ there exists some $\alpha=\alpha(G\cdot\pi)\in [0,k]$ such that:
    \begin{align*}
        \dim\pi g\nu & = \alpha \quad\text{for $\lambda_G$-almost all $g\in G$}, \\
        \dim\pi g\nu & \ge \alpha \quad\text{for all $g\in G$}.
    \end{align*}
\end{theorem}
Our approach borrows many ideas from that of Falconer-Jin \cite{Falconer2014Jin}, but we give a self-contained presentation for the reader's benefit, and since there are some significant differences. Notably, we do not assume that $G$ is connected, and we are able to achieve some simplifications.  When $G$ is connected it is not too hard to deduce Theorem \ref{thm:constant-dim} from the approach of \cite{Falconer2014Jin}; nonetheless, even in that case it was not stated explicitly to our knowledge.

\begin{remark}
    It remains an interesting problem whether $\dim \pi g\nu=\alpha$ for \emph{all} $g\in G$. This is the case under the assumptions of Theorem \ref{Main Tech theorem}, but we do not know the answer in general.
\end{remark}

Throughout this section, we fix a weighted IFS
\[
    \mathbf{\Lambda} = \big\lbrace f_i(x)=r_i h_i(x)+t_i; p_i \big\rbrace_{i=1}^n.
\]
Let $\nu$ be the associated self-similar measure. Denote by $G\le \mathbb{O}(\R^d)$ the closed group generated by $h_1,\ldots,h_n$, and let $\lambda_G$ be its normalized Haar measure.
\subsection{Some preliminary lemmas}

Write $\mathcal{A}=\{1,\ldots,n\}$. Let $\sigma$ be the left-shift on $(\mathcal{A}^{\N},p^{\N})$, and let $F: \mathcal{A}^{\N}\times G\to  \mathcal{A}^{\N}\times G$ be the skew product
$$(\omega,g)\to (\sigma\omega, h_{\omega_1}g).$$
Clearly $F$ preserves the measure $p^{\N}\times \lambda_G$ in the sense that
$$F\left( p^{\N}\times \lambda_G \right) = p^{\N}\times \lambda_G.$$
In fact, more can be said about this measure preserving system.  For a standard reference on ergodic theory see \cite{einsiedler2011ergodic}.
\begin{lemma} \label{lem:ergodicity}
    The measure preserving system
    $$(\mathcal{A}^{\N}\times G, F, p^{\N}\times \lambda_G)$$
    is ergodic.
\end{lemma}
\begin{proof}
    By \cite[Corollary 4.5]{parry1997}, it is enough to check that $F$ has a dense orbit. But this is clear since $G$ is by definition the closure of the group generated by the $h_i$, and the base map is the full shift. See e.g. \cite[Proposition 2.3]{Falconer2014Jin} for more details.
\end{proof}

Let $\mathcal{A}^*$ denote the family of finite words in the alphabet $\mathcal{A}$. For a finite word $(u_1,\ldots,u_j)\in \mathcal{A}^*$, let
\[
    p_u = p_{u_1}\cdots  p_{u_j},\quad  f_u = f_{u_1}\cdots f_{u_j},\quad  h_u = h_{u_1}\cdots h_{u_j},\quad r_u = r_{u_1} \cdots r_{u_j}
\]
Let $\overline{r}=\max\{r_1,\ldots,r_n\}\in (0,1)$. Given $q\in\mathbb{N}$, let
\[
    \mathcal{A}_q = \big\lbrace (u_1,\ldots,u_k)\in\mathcal{A^*}: r_{u_1\cdots u_{k-1}}>\overline{r}^q,\, r_{u_1\cdots u_k} \le \overline{r}^q  \big\rbrace \subset \mathcal{A}^*.
\]
This is a cut-set of the tree $\mathcal{A}^{\N}$, in the sense that every infinite word in $\mathcal{A}^{\N}$ starts with exactly one word in $\mathcal{A}_q$. It follows that $\nu$ is also the self-similar measure associated to the WIFS
\[
    \mathbf{\Lambda}_q = \{ f_u, p_u\}_{u\in\mathcal{A}_q}.
\]
We denote by $\eta_q$ the Bernoulli measure on $\mathcal{A}_q^{\N}$ with weights $\{ p_u:u\in\mathcal{A}_q^{\N}\}$.

Let $G_q$ be the closed group generated by the orthogonal parts of the elements of $\mathbf{\Lambda}_q$. It is clear that $G_q\le G$ for all $q$, but $G_q$ may vary with $q$. However, we have the following:
\begin{lemma} \label{lem:groups-Gq}
    Each $G_q$ is a union of some connected components of $G$. As $q$ varies in $\N$ there are only finitely many options for the group $G_q$. In particular, there is an infinite set  $J\subset\N$ such that $G_q:=G'$ for all $q\in J$.
\end{lemma}
\begin{proof}
    Since $G$ is a Lie group, the identity component $G^{\circ}$ is a normal clopen Lie subgroup of $G$, and the quotient $G/G^{\circ}$ is finite since $G$ is compact (its elements are the connected components). See e.g. \cite[Proposition 9.17]{baker2002}.

    Fix $q$ and let $H_q$ be the (non-closed) semigroup generated by $\{ h_u: u\in \mathcal{A}_q\}$. Then
    \[
        H_1=\bigcup\{ h_u\cdot H_q  : r_u> \overline{r}^q\}.
    \]
    Thus, a finite union of translates of $H_q$ is dense in $G$. By the Baire category theorem, there is $h\in H_1$ such that $h\cdot G_q$ has nonempty interior in $G$, whence $G_q$ itself has nonempty interior. It follows that $(G_q)^{\circ}=G^{\circ}$ (see e.g. \cite[Proposition 9.21]{baker2002}), and $G_q/G^{\circ}$ is a subset of the finite set $G/G^{\circ}$. The claim follows.
\end{proof}

For each $\pi\in \Gamma_{d,k}$, let $\{ \mathcal{D}_r^{\pi}\}$ be a partition of $\pi$ into half-open cubes of side length $r$. We can choose it to vary smoothly in $\pi$ and $r$. Given $\mu\in\mathcal{P}(\pi)$, let $H_r(\mu)$ denote its Shannon entropy  with respect to $\mathcal{D}_r^{\pi}$:
\[
    H_r(\mu) = \sum_{Q\in\mathcal{D}_r^{\pi}} \mu(Q)\log (1/\mu(Q)),
\]
with the usual interpretation $0\log 0=0$. Although $H_r$ is not continuous in the weak$^*$ topology, it has continuous modifications:
\begin{lemma} \label{lem:cont}
    For each $r>0$, there exists a function $\widetilde{H}_r:\mathcal{P}(\pi) \rightarrow \mathbb{R}$ such that:
    \begin{enumerate}
        \item It is continuous in the weak$^*$ topology; and
        \item  $|H_r(\mu)-\widetilde{H}_r(\mu)|=O_{k}(1)$ for every  $\mu\in\mathcal{P}(\pi)$.
    \end{enumerate}
\end{lemma}
\begin{proof}
    Let $\psi:\pi\to [0,1]$ be a smooth bump function such that $\psi(x)=1$ for $|x|\leq 1$ and $\psi(x)=0$ for $|x|\geq 2$. Define
    \[
        \widetilde{H}_r(\mu) := \int \log\left( \frac{1}{\int \psi(x+ry)d\mu(y)} \right)  d\mu(x),
    \]
    with the understanding that the integral is restricted  to values of $x$ for which the integrand is finite. Then $\widetilde{H}_r$ is weak$^*$ continuous, and it follows from \cite[Lemma 2.3]{peres2000existence} that it satisfies the second property.
\end{proof}

In the next Lemma, we investigate the entropy of projections of small pieces of $\nu$. To be more precise, for a word $u\in \mathcal{A}_q ^*$, let $[u]$ be the cylinder of infinite words starting with $u$. Let
\[
    \eta_{q,u} := \frac{1}{\eta_q[u]} (\eta_q)|_{[u]}
\]
be the normalized restriction of $\eta_q$ to $[u]$ (recall that $\eta_q$ is the Bernoulli measure corresponding to $\mathbf{\Lambda}_q$). Let $\Phi_q$ be the coding map corresponding to $\mathbf{\Lambda}_q$ (see \eqref{eq:coding-map}). We have the following:
\begin{lemma} \label{lem:scaled-entropy}
    For any $g\in G_q$ and any $u\in \mathcal{A}_q^m$,
    \[
        H_{{\overline{r}^{q(m+1)}}}(\pi g\Phi_q \eta_{q,u}) \ge H_{{\overline{r}^q}}(\pi h_{q,u}g \nu) - O_k(1).
    \]
\end{lemma}
\begin{proof}
    By the definition \eqref{eq:coding-map} of the coding map, for any $\omega\in\mathcal{A}_q^{\N}$,
    \[
        \Phi_q(u\omega) = r_{q,u} h_{q,u} \Phi_q(\omega) + t_{q,u}.
    \]
    Here the sub-index $q$ denotes that the quantities are associated to the IFS $\mathbf{\Lambda}_q$, and $t_{q,u}$ is some translation which is not important in what follows. We deduce that
    \[
        \pi\Phi_q(\eta_{q,u}) = r_{q,u} \cdot (\pi h_{q,u}) \nu + \pi t_{q,u},
    \]
    where abusing notation slightly, $r\cdot \mu$ denotes the measure $\mu$ scaled by $r$. Note that $r_{q,u}\le\overline{r}^{qm}$. Rescaling by $r_{q,u}^{-1}$ and applying standard properties of entropy (see \cite[\S 3.1]{hochman2014self}), we obtain
    \begin{align*}
        H_{\overline{r}^{q(m+1)}}\left(\pi \Phi_q\eta_{q,u} \right) & \ge H_{\overline{r}^{q(m+1)}}\left(r_{q,u}\cdot (\pi h_{q,u})\nu\right) - O_k(1) \\
                                                                    & \ge H_{\overline{r}^q}(\pi h_{q,u} \nu) - O_k(1)        .
    \end{align*}

    Since $g\nu$ is the self-similar measure for the conjugate IFS $\{ g f_u g^{-1}: u\in\mathcal{A}_q\}$,  with coding map $g\Phi_q$, the same argument applies in the general case.
\end{proof}

\subsection{Local entropy averages}

It is well known (see \cite[Theorem 1.3]{Fan2002measures}) that for any $\mu\in\mathcal{P}(\pi)$,
\begin{equation} \label{eq:entropy-vs-hausdorff}
    \dim\mu \le \liminf_{r\to 0} \frac{H_r(\mu)}{\log(1/r)}.
\end{equation}
The expression on the right-hand side is known as the (lower) entropy dimension of $\mu$. The \emph{local entropy averages} approach of Hochman-Shmerkin \cite{hochman2009local} provides a \emph{lower} bound for $\dim\pi\nu$ in terms of averages of the entropies of the projections of small pieces of $\nu$. The next Lemma formulates this in our context. To state it, we need some definitions from \cite[Section 4 and 5]{hochman2009local}.

\begin{definition}
    \begin{enumerate}
        \item   A \emph{tree} is an infinite graph with no cycles, bounded degree, and a distinguished vertex called the root. We identify the tree with infinite paths emanating from the root.

        \item     A \emph{tree morphism} is a continuous map between trees which preserves the root and the adjacency relation.

        \item Given a tree $X$ and $\rho>0$, we define a metric $d_{\rho}$ on $X$ by
              \[
                  d_{\rho}(x,y) = \rho^{n(x,y)},
              \]
              where $n(x,y)$ is the length of the longest common path from the root that is common to $x$ and $y$. We call $(X,d_{\rho})$ a \emph{$\rho$-tree}.

    \end{enumerate}

\end{definition}

We will apply the following Lemma to $F=\pi g\Phi_q$, but we state it more generally for clarity.
\begin{lemma}[Local entropy averages] \label{lem:local-entropy-averages}
    Let $q>0$ and   let us consider $\mathcal{A}_q^{\N}$ as an $\overline{r}^q$-tree. Let $F:\mathcal{A}_q^{\N}\to \pi$,  $\pi\in \Gamma_{d,k}$, be a $1$-Lipschitz map. Suppose that, for $\eta_q$-almost all $\omega$,
    \begin{equation} \label{eq:local-entropy-assumption}
        \liminf_{N\to\infty} \frac{1}{N} \sum_{j=1}^{N} \frac{H_{\overline{r}^{q.(j+1)}}(F\eta_{q,(\omega_1,\ldots,\omega_j)})}{q\log(1/\overline{r})} \ge \alpha.
    \end{equation}
    Then,
    \[
        \dim(F\eta_q)\ge \alpha - \frac{O_{k,\overline{r}}(1)}{q}.
    \]
\end{lemma}
\begin{proof}
    Write $\rho=\overline{r}^q$ for simplicity. Recall the definition of $C$-faithful maps from \cite[Section 5]{hochman2009local}. By \cite[Theorem 5.4]{hochman2009local}, we can factor $F= \Psi\circ F'$, where $X$ is a $\rho$-tree, $F':\mathcal{A}_q^\N\to X$ is a tree morphism, and $\Psi:X\to \R^k$ is a $C$-faithful map, for some $C$ depending only on $k$. Moreover,
    \[
        \left| H_{\rho^{k+1}}(F\eta_{q,u}) -  H_{\rho^{k+1}}(F'\eta_{q,u}) \right| = O_k(1).
    \]
    It then follows from our assumption \eqref{eq:local-entropy-assumption} that
    \[
        \liminf_{N\to\infty} \frac{1}{N} \sum_{j=1}^{N} \frac{H_{\rho^{j+1}}(F'\eta_{q,(\omega_1,\ldots,\omega_j)})}{q\log(1/\overline{r})} \ge \alpha - O_{\overline{r},k}(1/q).
    \]
    Applying \cite[Theorem 4.4]{hochman2009local}, we deduce that
    \[
        \dim F'\eta_q \ge \alpha - O_{\overline{r},k}(1/q).
    \]
    Here $\dim$ is with respect to the $\rho$-tree metric. Finally, using that $\Psi$ is $C$-faithful, and \cite[Proposition 5.2]{hochman2009local}, we conclude that
    \[
        \dim F\eta_q = \dim \Psi (F'\eta_q) \ge \alpha - O_{\overline{r},k}(1/q) - O_{C,k,\overline{r}}(1/q).
    \]
    This was our claim.
\end{proof}

Combining the above estimate with the ergodic theorem, we obtain the following key result towards the proof of Theorem \ref{thm:constant-dim}.
\begin{proposition} \label{prop:application-local-entropy-avg}
    Let $\pi \in \Gamma_{d,k}$, $q>1$, and let
    \[
        E_q =E_q(G_q\cdot\pi):= \int \frac{\widetilde{H}_{{\overline{r}^q}}(\pi h\nu)}{q\log(1/\overline{r})} \, d\lambda_{G_q}(h).
    \]
    Then
    \[
        \dim\pi g\nu \ge  E_q - \frac{C}{q},
        \quad\text{for all } g\in G_q.
    \]
    The constant $C$ depends only on $k$ and the IFS $\mathbf{\Lambda}$; in particular, it does not depend on $q$.
\end{proposition}
\begin{proof}
    To begin, note that by the translation-invariance of Haar measure, indeed $E_q$ depends only on the orbit $G_q\cdot\pi$ and not the specific projection $\pi$ in this orbit.

    By Lemma \ref{lem:ergodicity} applied to $\mathbf{\Lambda}_q$, the product measure $\eta_q\times \lambda_{G_q}$ is ergodic under the skew product $F_q:(\omega,g)\mapsto (\sigma\omega, h_{\omega_1}g)$. By \cite[Lemma 3.3]{Galicer2016Lq}, for all $\eta_q$-generic points $\omega$ (in particular, for $\eta_q$-almost all $\omega$),
    \begin{equation*} \label{eq:entropy-avg-ergodic}
        \lim_{N\to\infty} \frac{1}{N} \sum_{j=1}^{N} \frac{\widetilde{H}_{{\overline{r}^{q}}}(\pi h_{\omega_1}  \cdots h_{\omega_j} g\nu)}{q \log(1/\overline{r})} = E_q, \quad\text{uniformly in } g\in G_q.
    \end{equation*}
    Indeed, this holds since the function $(\omega,g)\to \widetilde{H}_{{\overline{r}^q}}(\pi g\nu)$ is continuous by Lemma \ref{lem:cont}. Since, by Lemma \ref{lem:cont} again, $H_{\rho}$ and $\widetilde{H}_{\rho}$ differ by a constant, we deduce from Lemma \ref{lem:scaled-entropy} that for $\eta_q$-almost all $\omega$,
    \[
        \liminf_{N\to\infty}\frac{1}{N}\sum_{j=1}^N \frac{H_{\overline{r}^{q(j+1)}}(\pi g \Phi_q \eta_{q,(\omega_1,\ldots,\omega_j)})}{q\log(1/\overline{r})} \ge E_q - O_{k,\overline{r}}(1/q),
    \]
    uniformly in $g\in G_q$. The function $\pi g\Phi_q$ is Lipschitz, and we can make it $1$-Lipschitz by a harmless rescaling of the original IFS. Since $\nu=\Phi_q \eta_q$, the conclusion follows from Lemma \ref{lem:local-entropy-averages}.
\end{proof}

\subsection{Proof of Theorem \ref{thm:constant-dim}}
We are now able to deduce the main result of this Section, Theorem \ref{thm:constant-dim}. We follow the proof of \cite[Theorem 8.2]{hochman2009local} with suitable modifications. Let $J\subseteq \mathbb{N}$ be the index set given in Lemma \ref{lem:groups-Gq}. Let $G'=G_q$, $q\in J$, be the  group $G_q$ that does not depend on $q$, provided by the Lemma.

It follows from Proposition \ref{prop:application-local-entropy-avg} that
\[
    \dim\pi g \nu \geq \limsup_{q\in J} E_q \quad\text{for all }g\in G'.
\]
On the other hand, by Fatou's Lemma and \eqref{eq:entropy-vs-hausdorff}
\begin{align*}
    \liminf_{q\in J} E_q \geq \int \liminf_{q\in J} \frac{\widetilde{H}_{{\overline{r}^q}}(\pi h\nu)}{q\log(1/\overline{r})} \, d\lambda_{G'}(h) \ge \int \dim \left( \pi h\nu \right) \, d\lambda_{G'}(h)=:\alpha.
\end{align*}
Combining the above inequalities, we obtain the claim with $G'$ in place of $G$. Since $G'\supset G^{\circ}$ and $\lambda_{G^{\circ}}\ll \lambda_{G'}$, the claim also holds with $G^{\circ}$ in place of $G$; we will now extend it to $G$.

Recall that $G/G^{\circ}$ is the finite set of connected components of $G$. As usual, let $H$ be the semigroup generated by $h_1,\ldots,h_n$. Since $H$ is dense in $G$, it intersects all connected components. Fix such a component $G^{\circ}\cdot h$, $h\in H$. By self-similarity of $\nu$, there exist $p\in (0,1)$ and a measure $\nu'$ such that
\[
    \nu = p h \nu + (1-p)\nu'.
\]
In particular, $h\nu \ll \nu$, and therefore $\pi g h\nu \ll \pi g\nu$ for all $g\in G^{\circ}$. By definition of Hausdorff dimension, this implies that
\[
    \dim \pi g h\nu \geq \dim \pi g\nu \ge \alpha, \quad g\in G^{\circ}.
\]
Applying the same reasoning to $h^{-1}\nu$ in place of $\nu$, and using that $H^{-1}$ also intersects all connected components, we obtain
\[
    \dim \pi g h\nu \leq \dim \pi g\nu = \alpha \quad\text{for $\lambda_{G^{\circ}}$-a.e. $g\in G^{\circ}$}.
\]
Combining the last two displayed equations completes the proof. \hfill{$\Box$}

\section{Proof of Theorem \ref{Main Tech theorem} and its corollaries} \label{Section proof}

\begin{proof}[Proof of Theorem \ref{Main Tech theorem}]

    Let $\nu$ be a self-similar measure on $\mathbb{R}^d$, and as usual let $G$ denote the closed group generated by the orthogonal parts of the functions in the IFS. Fix $\pi \in \Gamma_{d,k}$.

    It was proved by Feng-Hu \cite[Theorem 2.8]{feng2009dimension} that $\nu$ is exact dimensional, that is,
    \[
        \lim_{r\to 0} \frac{\log\nu(B(x,r))}{\log r} = \dim \nu \quad \text{for $\nu$-a.e. } x\in\R^d.
    \]
    Note that $B(x,r)\subset \pi^{-1}\pi(B(x,r))$ and therefore $\nu(B(x,r))\le \pi\nu(B(\pi x,r))$. It follows that
    \[
        \limsup_{r\to 0} \frac{\log \pi\nu(B(y,r))}{\log r} \le \min \{ k, \dim \nu\}  \quad \text{for $\pi \nu$-a.e. } y\in \pi.
    \]
    Hence, it is enough to show that the value of $\alpha$ in Theorem \ref{thm:constant-dim} equals $\min \{ k, \dim \nu\}$. For this, it is enough to show that
    \begin{equation} \label{eq:dim-claim-to-prove}
        \dim \pi g\nu = \min \lbrace k,\, \dim \nu \rbrace \quad\text{for $\lambda_G$-almost all } g\in G.
    \end{equation}

    Let $\gamma:(-\e,\e)\to\R^d$ be a $G\cdot\pi$-adapted (in particular, non-degenerate) curve.  Restricting the domain if needed, we can write $\gamma = \sigma\cdot \pi$ for some smooth curve $\sigma:(-\e,\e)\to G$. Note that $g\cdot \gamma = g\sigma\cdot \pi$ is also $G\cdot\pi$-adapted for all $g\in G$.

    As in Section \ref{Section GGW}, let $\Pi_{\theta}^{(k)}$ denote the plane spanned by the first $k$ derivatives of $\gamma$. By Theorem \ref{Theorem Gan measures}, for each $g\in G$,
    \begin{equation} \label{eq:application-GGW}
        \dim \Pi_\theta^{(k)}g \nu = \min \lbrace k,\, \dim \nu \rbrace \quad \text{for $m$-a.e. } \theta\in (-\e,\e),
    \end{equation}
    where $m$ is normalized Lebesgue measure. The push-forward of $m\times \lambda_G$ under $(\theta,g)\mapsto g\sigma(\theta)$ is invariant under left-translation and so equals $\lambda_G$. Since the osculating plane for $g\cdot \gamma=g\sigma\cdot \pi$ at $\theta$ is $g\cdot \Pi_\theta^{(k)}$, combining \eqref{eq:application-GGW} with Fubini's Theorem, we conclude that \eqref{eq:dim-claim-to-prove} holds. The proof is complete.
\end{proof}

$$ $$

\begin{proof}[Proof of Corollary \ref{cor:proj-dim-codim-1}]

    To begin, recall the Frenet-Serret formula of a curve $\gamma:(-\e,\e)\rightarrow\R^d$ parametrized by arc length, such that $\{ \gamma^{(j)}\}_{j=1}^{d-1}$ are linearly independent:
    \begin{equation} \label{eq:FrenetSerret}
        \begin{bmatrix}
            e_1'   \\
            e_2'   \\
            e_3'   \\
            \vdots \\
            e_d'
        \end{bmatrix}
        =
        \begin{bmatrix}
            0         & \kappa_1  & 0        & \cdots        & 0            \\
            -\kappa_1 & 0         & \kappa_2 & \cdots        & 0            \\
            0         & -\kappa_2 & 0        & \cdots        & 0            \\
            \vdots    & \vdots    & \vdots   & \ddots        & \kappa_{d-1} \\
            0         & 0         & 0        & -\kappa_{d-1} & 0
        \end{bmatrix}
        \begin{bmatrix}
            e_1    \\
            e_2    \\
            e_3    \\
            \vdots \\
            e_d
        \end{bmatrix}
    \end{equation}
    Here, $(e_1,\ldots,e_d)$ is the Frenet-Serret frame of $\gamma$, and $(\kappa_j)_{j=1}^{d-1}$ are the curvature scalars. See \cite[\S 1.7]
    {abatetovena2012} for more details.

    Consider first the projection onto $\pi$. Let $v$ be a unit vector generating $\pi$. Consider the orbit $G^{\circ}.v$ of the action of $G^{\circ}$ on $S^{d-1}$; this is a smooth connected manifold that, by assumption, is not contained in any linear hyperplane. Let $\dot\gamma:\R\rightarrow G^{\circ}.v$ be any smooth curve in $G^{\circ}.v$ parametrized by arc-length that is not contained in any linear hyperplane (e.g. take a smooth, unit speed curve through $d$ linearly independent points). The curve $\gamma$ itself is obtained by integrating $\dot\gamma$ starting from $0$, but it plays no role in the argument.

    We claim that $\gamma$ is non-degenerate. Arguing by induction in the first dimension $\le d$ for which $\gamma$ is degenerate, we may assume that the first $d-1$ derivatives of $\gamma$ are linearly independent. Assume the torsion $\kappa_{d-1}$ is identically zero. Then $\dot e_d=0$ by \eqref{eq:FrenetSerret}, i.e. $e_d$ is constant. But $\dot \gamma \cdot e_d=0$, so this contradicts the fact that $\dot \gamma$ is not contained in a linear hyperplane. Hence, $\kappa_{d-1}$ is not identically zero and (after passing to a sub-curve if needed) this gives non-degeneracy using \eqref{eq:FrenetSerret} again. It follows that $\gamma$ is $\pi G$-adapted.

    For the projection onto $\pi^\perp$, let $\gamma$ be the same curve as before, with Frenet-Serret frame $e_1,\ldots,e_d$, and curvatures $\kappa_1,\ldots,\kappa_{d-1}$, all of which are non-vanishing. Recall that all curvatures are positive except possibly $\kappa_{d-1}$. Replacing $\gamma$ by $-\gamma$ if needed (this does not affect any of the relevant properties of the curve), we may assume that also $\kappa_{d-1}>0$. It is a simple matter to check that the orthonormal frame $\tilde{e}_i = (-1)^{i+1} e_{d+1-i}$ satisfies the Frenet-Serret equation \eqref{eq:FrenetSerret} with curvatures $\tilde{\kappa}_i = \kappa_{d+1-i}$. Let $\tilde{\gamma}$ be the curve with Frenet-Serret frame $\tilde{e}_1,\ldots,\tilde{e}_d$ and curvatures $\tilde{\kappa}_1,\ldots,\tilde{\kappa}_{d-1}$ (the curve is obtained by integrating $\tilde{e}_1$). Then $\tilde{\gamma}$ is non-degenerate, and the linear span of its first $d-1$ derivatives is $e_1^\perp$. By construction, $e_1(s) \in G\pi$ for all $s$, and since $G\le \mathbb{O}(\R^d)$, also $e_1^{\perp}(\theta)\in G\pi^{\perp}$ for all $s$. It follows that $\tilde{\gamma}$ is $ G\cdot \pi^{\perp}$-adapted.

    In both cases, the claim now follows from Theorem \ref{Main Tech theorem}.

\end{proof}
$$ $$

\begin{proof}[Proof of Corollary \ref{cor:one-parameter}]

    Let $\dot\gamma(t)= e^{t A}v$. Then $\gamma^{(j)}(t) = e^{t A}A^{j-1} v$. The assumption that $v$ is cyclic for $A$ implies that $\gamma$ is non-degenerate. Since $\{ e^{t A}:t\in\R\}\subset G$, the curve $\gamma$ is also $G\cdot \pi_k$-adapted. The claim follows from Theorem \ref{Main Tech theorem}.
\end{proof}

\bibliography{bib}
\bibliographystyle{plain}
\end{document}